\documentclass[10pt]{article}
\usepackage{graphicx}
\usepackage{amsmath}
\usepackage{amsfonts}
\usepackage{amssymb}
\usepackage{amsthm} % when in use, disable the definition of the proof env.
\usepackage[english]{babel}
\usepackage{enumerate}
\usepackage[cp1250]{inputenc}
\usepackage[a4paper, left=2.5cm, right=2.5cm, top=3.5cm, bottom=3.5cm, headsep=1.2cm]{geometry}
\usepackage{authblk}
\usepackage[para]{footmisc}
\usepackage{caption}
\usepackage{color}
\definecolor{sepia}{cmyk}{0, 0.83, 1, 0.70}
\usepackage[svgnames]{xcolor}

\usepackage[plainpages=false,pdfpagelabels,bookmarksnumbered,%
 colorlinks=true,%
 linkcolor=sepia,%
 citecolor=sepia,%
 unicode]{hyperref}

\usepackage{makecell}
\usepackage{rotating}
\renewcommand\theadfont{\bfseries}

\newtheorem{theorem}{Theorem}

\newtheorem{corollary}[theorem]{Corollary}

\newtheorem{lemma}[theorem]{Lemma}

\newtheorem{proposition}[theorem]{Proposition}

\theoremstyle{definition}

\newtheorem{definition}[theorem]{Definition}
\newtheorem{example}[theorem]{Example}

\newtheorem{remark}[theorem]{Remark}

\numberwithin{equation}{section} % for equations
\makeatletter
\renewenvironment{proof}[1][\proofname]
{\par
	\pushQED{$\blacksquare$} % originally "\qed" qas instead of $\blacksquare$
	\normalfont\topsep6\p@\@plus6\p@\relax
	\trivlist
	\item[\hskip\labelsep\bfseries#1\@addpunct{.}]
	\ignorespaces}
{\popQED \endtrivlist\@endpefalse}
\makeatother

\DeclareMathOperator*{\interior}{int}
\DeclareMathOperator*{\ri}{ri}

\DeclareMathOperator*{\fix}{Fix}
\DeclareMathOperator*{\id}{Id}

\DeclareMathOperator*{\argmax}{argmax}

\DeclareMathOperator*{\supremum}{sup}

\begin{document}

\title{\textbf{Convergence Properties of Dynamic String Averaging Projection Methods in the Presence of Perturbations}}
\author[1,2]{Christian Bargetz\thanks{christian.bargetz@uibk.ac.at}}
\author[1]{Simeon Reich\thanks{sreich@tx.technion.ac.il}}
\author[1]{Rafa\l\ Zalas\thanks{rzalas@tx.technion.ac.il}}
\affil[1]{Department of Mathematics,\protect\\The Technion---Israel Institute of Technology,\protect\\32000 Haifa, Israel}
\affil[2]{Department of Mathematics,\protect\\University of Innsbruck,\protect\\Technikerstraße~13, 6020 Innsbruck, Austria}
\maketitle

\begin{abstract}
  Assuming that the absence of perturbations guarantees weak or strong convergence to a common fixed point, we study the behavior of perturbed products of an infinite family of nonexpansive operators. Our main result indicates that the convergence rate of unperturbed products is essentially preserved in the presence of perturbations. This, in particular, applies to the linear convergence rate of dynamic string averaging projection methods, which we establish here as well. Moreover, we show how this result can be applied to the superiorization methodology.
\end{abstract}

\section{Introduction}
For a given \textit{common fixed point problem}, that is, find  $x\in \bigcap_{k=0}^\infty \fix T_k$,
we consider an iterative scheme of the following form:
\begin{equation}\label{eq:ItScheme}
  x^{0}\in\mathcal{H}\qquad\text{and}\qquad x^{k+1} := T_kx^k,
\end{equation}
where each $T_k\colon \mathcal{H}\to\mathcal{H}$, $k=0,1,2,\ldots$, is a nonexpansive operator  and $\mathcal H$ is a Hilbert space. We recall that $T\colon\mathcal H\rightarrow \mathcal H$ is \textit{nonexpansive} if $\|Tx-Ty\| \leq \|x-y\|$
for all $x,y\in\mathcal{H}$. 
Scheme~\eqref{eq:ItScheme} is called \emph{perturbation resilient} if under the assumption that it generates a sequence converging to a solution of the problem, it follows that 
any sequence $\{y^k\}_{k=0}^{\infty}$ satisfying
\begin{equation}\label{eq:PertSeq}
  y^0\in\mathcal{H} \qquad\text{and}\qquad \|y^{k+1}-T_ky^k\| \leq e_k,
\end{equation}
where $\{e_k\}_{k=0}^{\infty}$ is a suitable sequence of errors, also converges to a (possibly different) solution of the problem.

Perturbation resilience of iterative projection methods was first studied in~\cite{Com2001NumericalRobustness}, where the sequence of errors is assumed to be summable. In~\cite{ButnariuReichZaslavski2008} both weak and strong convergence properties of infinite products of nonexpansive operators in the presence of summable perturbations are considered. As an example of a result on perturbation resilience with respect to weak and strong convergence, we recall a variant of \cite[Theorems~3.2 and 5.2]{ButnariuReichZaslavski2008} in the Hilbert space setting.

\begin{theorem}\label{thm:WeakPertRes}
  For every $k=0,1,2,\ldots$, let $T_k\colon\mathcal H\rightarrow\mathcal H$ be nonexpansive and assume that $C\subseteq \bigcap_k \fix T_k$ is nonempty, closed and convex. Let $\{e_k\}_{k=0}^\infty\subseteq[0,\infty)$ be a given summable sequence of errors and let $\{x^k\}_{k=0}^\infty\subseteq\mathcal H$ be an inexact trajectory such that
  \begin{equation}\label{eq:th:GSC:ek}
    \|x^{k+1}-T_kx^k\|\leq e_k
  \end{equation}
  holds for each $k=0,1,2,\ldots$. Assume that for all $i=0,1,2,\ldots$, there is an $x_i^\infty\in C$ such that
  \begin{equation}\label{eq:WeakConvAssBRZ}
    T_k\cdots T_i x^i \rightharpoonup x_i^\infty.
  \end{equation}
  Then there is a point $x^\infty\in C$ such that $x^k\rightharpoonup x^\infty$. If, in addition, for all $i=0,1,2,\ldots$,
  \begin{equation}\label{eq:StrongConvAssBRZ}
    T_k\cdots T_ix^i \to x_i^\infty,
  \end{equation}
  then $x^k\to x^\infty$.
\end{theorem}
Clearly assumptions \eqref{eq:WeakConvAssBRZ} and \eqref{eq:StrongConvAssBRZ} fit into the perturbation resilience paradigm. The interpretation of these conditions is that if at some point, say $i$, we interrupt the perturbed process $\{x^k\}_{k=0}^\infty$ and begin exact computations starting from $x^i$, then our method will converge. We note here that originally Theorems 3.2 and 5.2 of \cite{ButnariuReichZaslavski2006} were formulated with stronger conditions in complete metric and Banach spaces, respectively. Nevertheless, the proofs remain to be true if one assumes only \eqref{eq:WeakConvAssBRZ} and \eqref{eq:StrongConvAssBRZ} (actually one should see the proofs of \cite[Theorems~3.1 and 5.1]{ButnariuReichZaslavski2008}). A simpler variant of Theorem \ref{thm:WeakPertRes} appeared in \cite[Theorems~4.1 and 4.2]{ButnariuReichZaslavski2006} with a constant sequence of operators ($T_k=T$, for some $T$).

The assumption of summable errors seems to be common for a variety of iterative schemes. For example, in~\cite{NR1997AccretiveOperators} perturbation resilience with respect to summable errors of an iterative scheme for finding zeros of an accretive operator is considered. On the other hand, summability of errors is a basic assumption for the quasi-Fej\'{e}r monotone sequences which constitute another important tool in the study of numerical robustness; see, for example, \cite{Combettes2001,Combettes2015}.

Besides taking into account computational errors, the main reason for the interest in perturbation resilience is the \emph{superiorization methodology}. Roughly speaking, the idea behind this methodology is to use perturbation resilience of an iterative method to introduce perturbations which steer the sequence towards a limit which not only solves the original problem, but should also be superior with respect to the solution obtained without perturbations.
 In this context, iterative schemes of the form
\begin{equation}
  x^0\in\mathcal{H}\qquad\text{and}\qquad x^{k+1} := T_k(x^k-\beta_kv^k),
\end{equation}
are considered where the sequence $\{\beta_k\}_{k=0}^{\infty}\subset(0,\infty)$ is summable and the steering sequence $\{v^k\}_{k=0}^{\infty}$ is bounded. In practice each $v^k$ is somehow related to the subgradient of a certain convex, continuous function $\phi\colon\mathcal H\rightarrow \mathbb R$, for example, $v^k\in\partial \phi(x^k)$, whereas ``superior'' is interpreted as a smaller value of $\phi$. 
For an introduction to superiorization, we refer the interested reader to~\cite{Cen2014Superiorization}. Applications of the superiorization methodology include optimization, see, for example, \cite{CDHST2013ProjectedSubgradient}, and image reconstruction, see, for example, \cite{CensorDavidiHerman2010, GHD2011Reconstruction, Her2009ImageReconstruction, HD2009ImageReconstruction, SP2008ImageReconstruction}.

String averaging projection methods have been introduced in~\cite{CEH2001StringAveraging} for solving the convex feasability problem, that is, given closed and convex sets $C_i$, $i=1,\ldots,M$, such that $C:=\bigcap_{i=1}^{M} C_i$ is nonempty, find a point $x^\infty\in C$. We denote by $P_{C_i}$ the metric projection onto $C_i$, that is, the mapping which maps a point $x\in\mathcal{H}$ to the unique point in $C_i$ with minimal distance to $x$. For $n=1,\ldots,N$, let $J_n=(j_1^n,\ldots,j_{|J_n|}^n)$ be a finite ordered subset of $\{1,\ldots, M\}$ called a \emph{string}. In addition, let $\omega_n\in(0,1)$, $n=1,\ldots,N$, satisfy $\sum_{n=1}^{N} \omega_n = 1$.
The \emph{string averaging projection method} for these data is the iterative method defined by
\begin{equation}\label{eq:DStringAv}
  x^0 \in \mathcal{H}\qquad\text{and}\qquad x^{k+1} := \sum_{n=1}^{N} \omega_n \prod_{j\in J_n} P_{C_j}x^k,
\end{equation}
where $x^0$ is an arbitrary initial point and $\prod_{j\in J} P_{C_j} := P_{C_{j_l}}\ldots P_{C_{j_1}}$ for a string $J=(j_1,\ldots,j_l)$. 
In~\cite{CEH2001StringAveraging} it is shown that any sequence $\{x^k\}_{k=0}^{\infty}$ generated by the above method converges to a point $x^\infty\in C\subseteq\mathbb{R}^{n}$. In addition to metric projections, in~\cite{CEH2001StringAveraging} also relaxed metric projections and Bregman projections are considered.

In~\cite{CZ2013PerturbationStringAveraging} a modification of method~\eqref{eq:DStringAv} is introduced. Instead of applying the same operator at each iterative step, both the strings and the weights may be different at each step. In more detail, the \emph{dynamic string averaging (SA) projection method} is defined by
\begin{equation}
  x^0 \in \mathcal{H}\qquad\text{and}\qquad x^{k+1} := T_k x^{k},
\end{equation}
where
\begin{equation}
  T_k:=\sum_{n=1}^{N_k} \omega^k_n \prod_{j\in J^k_n} P_{C_j}x^k,
\end{equation}
and for all $k=0,1,2,\ldots$, $J_n^k\subseteq I$ is a string and $\omega^k_n\in(0,1)$ with $\sum_{n=1}^{N_k}\omega^k_n=1$. Under the assumption that the family $ \{C_1,\ldots, C_M\}$ is boundedly regular and that the control is 1-intermittent, that is, at all steps $k$ each index $i\in\{1,\ldots,M\}$ appears in at least one of the strings $J_n^k$, the authors of~\cite{CZ2013PerturbationStringAveraging} show that any sequence generated by the dynamic string averaging projection method converges in norm to an element $x^\infty\in C\subseteq\mathcal{H}$. In addition, they prove that the superiorized version has the same convergence properties.

A static version of the superiorized SA projection method appeared for the first time in \cite{ButnariuDavidiHermanKazantsev2007} although it relies heavily on \cite{ButnariuReichZaslavski2006}. Many variants of the SA methods with operators more general than the metric projection can be found in the literature. For example, static SA methods based on the averaged operators in Hilbert space can be found in  \cite[Section 6]{ButnariuReichZaslavski2008}, whereas their dynamic variants (with $s$-intermittent control, $s\geq 1$) appeared in \cite{AleynerReich2008} and \cite[Corollary 5.18]{BauschkeCombettes2011}. 
In~\cite{ReichZalas2016} a general framework for the study of a modular SA methods based on cutters and firmly nonexpansive operators in Hilbert spaces is introduced and convergence properties, including perturbation resilience and superiorization, are studied. In particular, if the family of the fixed point sets is boundedly regular and the operators satisfy certain regularity assumptions, \cite[Theorem~4.5]{ReichZalas2016} shows perturbation resilience for the weak and strong convergence of these methods under summable perturbations.

In all of the above-mentioned results it is shown that weak or strong convergence holds true and that in some cases, the type of convergence can be preserved under summable perturbations. The focus of the present article is on \emph{preservation of the convergence rate}. We are interested, in particular, in the case of linear convergence, which is known to be the case for some variants of the SA projection methods such as cyclic and simultaneous projetion methods; see, for example, \cite{Aronszajn1950, BauschkeBorwein1996, BauschkeNollPhan2015, BeckTeboulle2003, DeutschHundal2006a, DeutschHundal2006b, DeutschHundal2008, GurinPoljakRaik1967, KolobovReichZalas2016, KayalarWeinert1988, Pierra1984}. We emphasize that to the best of our knowledge, linear convergence rates for static/dynamic SA projection methods were unknown till now.
We recall that a sequence $\{x^k\}_{k=0}^{\infty}$ in $\mathcal{H}$ \emph{converges linearly} to a point $x^{\infty}\in\mathcal{H}$ if there are constants $c>0$ and $0<q<1$ such that
\begin{equation}
\|x^k-x^{\infty}\| \leq c q^k
\end{equation}
for all $k=0,1,2,\ldots$.

Our paper is structured as follows.
In Section~\ref{sec:Preliminaries} we introduce some notation and, for the convenience of the reader, we present the concepts and properties we use in the other sections of this article.
In Section~\ref{sec:StringAveraging} we consider the convergence properties of the dynamic string averaging projection method and investigate their relation with the regularity of the family $\{C_1,\ldots,C_M\}$. In particular, we prove linear convergence provided the family
$\{C_1,\ldots,C_M\}$ is boundedly linearly regular. In Section~\ref{sec:InexactOrbits} we provide general results regarding the preservation of the convergence rate for infinite products of nonexpansive operators in the presence of summable perturbations. We specialize our main result to the preservation of linear convergence rates and to the case of superiorization.
In Section~\ref{sec:StringAveraging2} we combine the results of Sections~\ref{sec:StringAveraging} and~\ref{sec:InexactOrbits} to discuss the behavior of dynamic string averaging methods in the presence of summable perturbations. As a particular case, we consider the superiorized dynamic string averaging projection method.

\section{Preliminaries}\label{sec:Preliminaries}
In this paper $\mathcal{H}$ always denotes a real Hilbert space. For a sequence $\{x^k\}_{k=0}^\infty$ in $\mathcal{H}$ and a point $x^\infty\in\mathcal{H}$, we use the notation
\begin{equation}
  x^k\rightharpoonup x^\infty \qquad\text{and}\qquad x^k\to x^\infty
\end{equation}
to indicate that $\{x^k\}_{k=0}^{\infty}$ converges to $x^\infty$ weakly and in norm, respectively. 

Given a nonempty, closed and convex set $C\subseteq\mathcal{H}$, we denote by $P_C\colon\mathcal{H}\to\mathcal{H}$ the \emph{metric projection onto $C$}, that is, the operator which maps $x\in\mathcal{H}$ to the unique point in $C$ closest to $x$. The operator $P_C$ is well defined for such sets $C$ and it is not difficult to see that it is nonexpansive; see, for example, \cite[Proposition 4.8]{BauschkeCombettes2011}, \cite[Theorem 2.2.21]{Ceg2012IterativeMethods} or \cite[Theorem 3.6]{GoebelReich1984}. 
We denote by
\begin{equation}
  d(x,C) := \inf \{\|x-z\|\mid z\in C\}
\end{equation}
the distance of $x$ to $C$. In addition, for $\varepsilon>0$, we define
\begin{equation}
  C_\varepsilon := \{x\in\mathcal{H}\mid d(x,C)\leq \varepsilon\}.
\end{equation}
The following lemma provides a connection between the metric projection onto a nonempty, closed and convex set $C$ and the metric projection onto the enlarged set $C_\varepsilon$. The statement of this lemma can be found in \cite[Proposition 28.10]{BauschkeCombettes2011} without a proof. We provide a proof here for the convenience of the reader.

\begin{lemma}\label{lem:DistToSet}
  Let $C\subseteq \mathcal H$ be nonempty, closed and convex, $\varepsilon\geq 0$ and let $x\in \mathcal H\setminus C_\varepsilon$. Then
  \begin{equation}
  P_{C_\varepsilon}x = \left(1-\frac{\varepsilon}{\|x-P_{C}x\|}\right)P_{C}x+\frac{\varepsilon}{\|x-P_Cx\|}x.
  \end{equation}
  Consequently,
  \begin{equation}
  d(x,C) = d(x,C_\varepsilon) + \varepsilon.
  \end{equation}
\end{lemma}

\begin{proof}
  The inequality $d(x,C) \leq d(x,C_\varepsilon)+\varepsilon$ follows from the properties of the Hausdorff distance.

  For the convenience of the reader we present this argument. For all points $y\in C_\varepsilon$, we have
  \begin{align}\nonumber
    d(x,C) & = \inf\{\|x-z\|\mid z\in C\} \leq \inf\{\|x-y\| + \|y-z\|\mid z\in C\}\\
           & \leq \|x-y\| + d(y,C) \leq  \|x-y\| + \varepsilon
  \end{align}
  and hence
  \begin{equation}
  d(x,C) \leq \inf\{\|x-y\|\mid y\in C_\varepsilon\} + \varepsilon = d(x,C_\varepsilon)+\varepsilon.
  \end{equation}
  In order to show the reverse inequality, note that $\|x-y\|> \varepsilon$ for all $y\in C$. We set $z=P_Cx$, where $P_C$ is the nearest point projection onto $C$. This implies that $\|z-x\|=d(x,C)$ and $\frac{\varepsilon}{\|x-z\|}<1$. We define a new point $z_\varepsilon$ by
  \begin{equation}
  z_\varepsilon := \left(1-\frac{\varepsilon}{\|x-z\|}\right)z+\frac{\varepsilon}{\|x-z\|}x
  \end{equation}
  and get
  \begin{equation}
  \|z-z_\varepsilon\| = \frac{\varepsilon}{\|x-z\|} \|x-z\| = \varepsilon,
  \end{equation}
  that is, $z_\varepsilon\in C_\varepsilon$ because $z\in C$. On the other hand, we get
  \begin{equation}
  \|z_\varepsilon-x\| = \left(1-\frac{\varepsilon}{\|x-z\|}\right)\|z-x\| = \|z-x\| - \varepsilon
  = d(x,C)-\varepsilon
  \end{equation}
  and therefore $d(x,C_\varepsilon)\leq \|z_\varepsilon-x\| = d(x,C)-\varepsilon$, which finishes the proof.
\end{proof}

The point $y=P_Cx$ is characterized by: $y\in C$ and
\begin{equation}\label{eq:MetrProj}
  \langle x-y, z-y\rangle \leq 0
\end{equation}
for all $z\in C$; see, for example, \cite[Theorem 3.14]{BauschkeCombettes2011}, \cite[Theorem 1.2.4]{Ceg2012IterativeMethods} or \cite[Proposition 3.5]{GoebelReich1984}.
A direct computation shows that inequality~\eqref{eq:MetrProj} is equivalent to
\begin{equation}\label{eq:MetrProjas1SQNE}
  \|y-z\|^2 \leq \|x-z\|^2- \|x-y\|^2
\end{equation}
for all $z\in C$.
As a generalization of this property, for $\rho\geq 0$, an operator $T\colon \mathcal{H}\to\mathcal{H}$ is called \emph{$\rho$-strongly quasi-nonexpansive} if it satisfies
\begin{equation}
  \|Tx-z\|^2 \leq \|x-z\|^2 - \rho \|x-Tx\|^2
\end{equation}
for all $x\in\mathcal{H}$ and all $z\in \fix T$, where $\fix T$ denotes the set of all fixed points of $T$. The equivalence of \eqref{eq:MetrProj} and \eqref{eq:MetrProjas1SQNE} has been extended to $\rho$-strongly quasi-nonexpansive operators and $\alpha$-relaxed cutters; see, for example, \cite[Theorem 2.1.39]{Ceg2012IterativeMethods}. For the convenience of the reader, we now recall a few important properties of $\rho$-strongly quasi-nonexpansive operators.

\begin{theorem}\label{thm:SQNECompConvComb}
  Let $U_i\colon\mathcal{H}\to\mathcal{H}$ be $\rho_i$-strongly quasi-nonexpansive, $i=1,\ldots,m$, where $\min_{i} \rho_i>0$. Then the following statements hold true:
  \begin{enumerate}[(i)]
    \item The composition $U:=U_m\cdots U_1$ is $\rho$-strongly quasi-nonexpansive, where $\fix U = \bigcap_{i=1}^{m}\fix U_i$ and
        \begin{equation}
        \rho:=\left( \sum_{i=1}^m \frac 1 {\rho_i}\right)^{-1} \geq  \frac{\min_{i} \rho_i}{m}.
        \end{equation}
    \item Given $\omega_1,\ldots, \omega_m$, where $\omega_i>0$ for $i=1,\ldots,m$ and $\sum_{i=1}^{m}\omega_i=1$, the operator $U := \sum_{i=1}^{m} \omega_i U_i$ is $\rho$-strongly quasi-nonexpansive, where again $\fix U = \bigcap_{i=1}^{m} \fix U_i$ and
        \begin{equation}
        \rho:=\frac{\sum_{i=1}^m\frac{\omega_i\rho_i}{\rho_i+1}}
        {\sum_{i=1}^m\frac{\rho_i}{\rho_i+1}}\geq \min_{i} \rho_i.
        \end{equation}
  \end{enumerate}
\end{theorem}

\begin{proof}
  See, for example,~\cite[Theorem 2.1.48]{Ceg2012IterativeMethods} and~\cite[Theorem 2.1.50]{Ceg2012IterativeMethods}. Statement (i) can also be deduced from \cite[Proposition 2.5]{CombettesYamada2015}.
\end{proof}

\begin{theorem}\label{thm:SQNE-Composition}
  Let $U_i\colon\mathcal{H}\to\mathcal{H}$ be $\rho_i$-strongly quasi-nonexpansive, $i=1,\ldots,m$, satifying $\bigcap_{i=1}^{m} \fix U_i \neq \emptyset$. Let $\rho:=\min_{1\leq i\leq m} \rho_i>0$ and $U:=U_mU_{m-1}\ldots U_1$. Let $x\in\mathcal{H}$ and $z\in\bigcap_{i=1}^{m} \fix U_i$ be arbitrary. Then,
  \begin{equation}
    \|Ux-z\|^2 \leq \|x-z\|^2 - \sum_{i=1}^{m} \rho_i \|Q_ix-Q_{i-1}x\|^2,
  \end{equation}
  where $Q_0:=\id$ and $Q_i := U_{i}U_{i-1}\ldots U_1$ for $i=1,\ldots, m$.
\end{theorem}

\begin{proof}
  See~\cite[Proposition~4.6]{CZ2014ApproximatelyShrinking}.
\end{proof}

Given a nonempty, closed and convex set $C\subseteq\mathcal{H}$, a sequence $\{x^k\}_{k=0}^\infty$ is called
\emph{Fej\'{e}r monotone  with respect to $C$} if the inequality
\begin{equation}
  \|x^{k+1}-z\| \leq \|x^k-z\|
\end{equation}
holds for all $z\in C$ and $k=0,1,2,\ldots$.

\begin{theorem}\label{thm:FejerMonotone}
  Let $C\subseteq\mathcal{H}$ be nonempty, closed and convex, and let $\{x_k\}_{k=0}^{\infty}$ be Fej\'{e}r monotone with respect to~$C$.
  \begin{enumerate}[(i)]
  \item \label{thm:FejerMonotone:Item1} If $\{x^{k}\}_{k=0}^\infty$ converges strongly to some point $x^\infty\in C$, then for every $k=0,1,2,\ldots$, we have $\|x^k-x^\infty\|\leq 2d(x^k,C)$.
  \item \label{thm:FejerMonotone:Item2} If there is a constant $q\in (0,1)$ such that $ d(x^{k+1}, C) \leq q d(x^k,C)$, then $\{x^k\}_{k=0}^{\infty}$ converges linearly to some point $x^\infty\in C$, that is, $\|x^k-x^\infty\| \leq 2d(x_0,C)q^{k}$ for all $k= 0,1,2,\ldots$.
  \item \label{thm:FejerMonotone:Item3} If a subsequence $\{x^{ks}\}_{k=0}^{\infty}$ of $\{x^{k}\}_{k=0}^{\infty}$ converges linearly to a point $x^\infty\in C$, that is, $\|x^{ks}-x^\infty\|\leq cq^k$, then the whole sequence $\{x^k\}_{k=0}^{\infty}$ converges linearly to $x^\infty$, that is,
    \begin{equation}
      \|x^k-x^\infty\|\leq \frac{c}{(\sqrt[\scriptstyle{s}]{q})^{s-1}} \left(\sqrt[\scriptstyle{s}]{q}\right)^k.
    \end{equation}
  \end{enumerate}
\end{theorem}

\begin{proof}
  See, for example,~\cite[Theorem~2.16]{BauschkeBorwein1996} and~\cite[Proposition~1.6]{BauschkeBorwein1996}.
\end{proof}

\begin{definition}\label{def:BR:set}
  Given a set $S\subseteq\mathcal{H}$, a family $\mathcal{C} = \{C_1,\ldots,C_M\}$ of convex and closed sets $C_i\subseteq\mathcal{H}$ with a nonempty intersection $C:=\bigcap_{i=1}^M C_i$ 
    is called
  \begin{enumerate}[(i)]
  \item \emph{regular over $S$} if for every $\{x^k\}\subseteq S$,
    \begin{equation}
      \max_{i=1,\ldots, M} d(x^k,C_i) \to_k 0\quad \Longrightarrow\quad d(x^k,C) \to_k 0;
    \end{equation}
  \item \emph{$\kappa_S$-linearly regular over $S$} if
    \begin{equation}
      d(x,C) \leq \kappa_S \max_{i=1,\ldots,M} d(x,C_i)
    \end{equation}
    holds for all $x\in S$ and some constant $\kappa_S>0$.
  \end{enumerate}
  We say that the family $\mathcal{C}$ is \emph{boundedly (linearly) regular} if it is ($\kappa_S$-linearly) regular over every bounded subset $S\subseteq\mathcal{H}$. We say that $\mathcal C$ is \emph{(linearly) regular} if it is ($\kappa_S$-linearly) regular over $S=\mathcal H$.
\end{definition}

Note that the constant $\kappa_S$ always satisfies $\kappa_S\geq 1$ since $C\subseteq C_i$ implies $d(x,C_i) \leq d(x,C)$ for $i=1,\ldots,M$. Below we recall some examples of regular families of sets \cite[Fact 5.8]{BauschkeNollPhan2015}.

\begin{example}\label{example:regularSets}
Let $C_i\subseteq\mathcal H$ and $\mathcal C$ be as in Definition \ref{def:BR:set}. 
Then the following statements hold:
\begin{enumerate}[(i)]
\item If $\dim\mathcal H<\infty$, then $\mathcal C$ is boundedly regular.
\item If $C_M\cap\interior\bigcap_{i=1}^{M-1}C_i\neq\emptyset$, then $\mathcal C$ is boundedly linearly regular.
\item If each $C_i$ is a half-space, then $\mathcal C$ is linearly regular.
\item If $\dim\mathcal H<\infty$, $C_i$ is a half-space for $i=1,\ldots,L$, and $\bigcap_{i=1}^LC_i \cap \bigcap_{i=L+1}^M\ri C_i\neq\emptyset$, then $\mathcal C$ is boundedly linearly regular.
\item If each $C_i$ is a closed subspace, then $\mathcal C$ is linearly regular if and only if $\sum_{i=1}^MC_i^\perp$ is closed.
\end{enumerate}
\end{example}
For a prototypical convergence result combining properties of Fej\'{e}r monotone sequences with regularity properties of sets, see, for example, \cite[Theorem~2.10]{Bau2001ProjectionAlgorithms}. For more information regarding regular families of sets see \cite{BauschkeBorwein1996, ReichZaslavski2014a, ReichZaslavski2014b}.

\section{String averaging projection methods}\label{sec:StringAveraging}

In this section we present sufficient conditions for weak, strong and linear convergence of dynamic string averaging projection methods depending on the regularity of the constraint sets.

\begin{lemma}\label{lem:StringAveraging}
  Let $C_i\subseteq\mathcal{H}$, $i\in I:=\{1,\ldots,M\}$, be closed and convex, and assume $C:=\bigcap_{i\in I} C_i \neq \emptyset$. Let $U$  be the string averaging projection operator defined by
  \begin{equation}\label{eq:SADef}
    U:=\sum_{n=1}^{N} \omega_n \prod_{j\in J_n} P_{C_j},
  \end{equation}
  where $J_n\subseteq I$ is a string (finite ordered subset) and $\omega_n\in(0,1)$, $\sum_{n=1}^N\omega_n=1$. Moreover, assume that $I=J_1\cup\ldots \cup J_N$. Then the following statements hold true:

  \begin{enumerate}[(i)]
  \item $U$ is nonexpansive.
  \item $U$ is $(1/m)$-strongly quasi-nonexpansive, where $m:=\max_{1\leq n\leq N} |J_n|$. Moreover, $\fix U=C$.
  \item For all $x\in \mathcal{H}$, $z\in C$ and $i\in I$, we have
    \begin{equation}\label{eq:SAPartialBound}
      d(x,C_i)^2 \leq \frac{2m}{\omega}\big( \|x-z\|^2-\|Ux-z\|^2\big) \leq \frac{2m}{\omega}\|Ux-x\| d(x,C),
    \end{equation}
    where $\omega:=\min_{1\leq n\leq N}\omega_n$.
  \item Consequently, if the family $\{C_i\mid i\in I\}$ is $\kappa$-linearly regular over $S\subseteq\mathcal{H}$, then, for all $x\in S$, we have
    \begin{equation}\label{eq:SAFullBound}
      \|Ux-x\| \geq \frac{\omega}{2m\kappa^2} \, d(x,C).
    \end{equation}
  \end{enumerate}
\end{lemma}

\begin{proof}
  Note that (i) follows from the fact that both compositions and convex combinations of nonexpansive operators
  are nonexpansive. Statement (ii) follows immediately from Theorem~\ref{thm:SQNECompConvComb}.

  In order to show (iii), let $x\in \mathcal{H}$ and $z\in C:=\bigcap_{i\in I} C_i$. Using the convexity of the function $x\mapsto \|x\|^2$ and Theorem~\ref{thm:SQNE-Composition}, we get
  \begin{align}\label{eq:SASQNE-Ineq} \nonumber
    \|Ux-z\|^2 & = \Big\| \sum_{n=1}^{N} \omega_n \prod_{j\in J_n} P_{C_j} x - z \Big\|^2
                 = \Big\| \sum_{n=1}^{N} \omega_n \Big(\prod_{j\in J_n} P_{C_j} x - z\Big) \Big\|^2 \\ \nonumber
               & \leq \sum_{n=1}^{N} \omega_n \Big\|\prod_{j\in J_n} P_{C_j}x -z\Big\|^2
                 \leq \sum_{n=1}^{N} \omega_n \Big(\|x-z\|^{2}-\sum_{l=1}^{m_n} \big\|Q_{l}^nx-Q_{l-1}^nx\big\|^2\Big)\\
               & = \|x-z\|^2 -\sum_{n=1}^{N}\sum_{l=1}^{m_n} \omega_n \big\|Q_l^nx-Q_{l-1}^nx\big\|^2,
  \end{align}
  where  $m_n:=|J_n|$, $J_n:=(i_{1},\ldots,i_{m_n})$, $Q_l^n:=\prod_{j=1}^{l} P_{C_{i_j}}$, and $Q_0^n:=\id$. Note that in the above estimates we also use the facts that $C\subseteq \bigcap_{j\in J_n}C_j$ and $\sum_{n=1}^N\omega_n=1$.
  From the above inequality, we deduce that
  \begin{equation}\label{eq:DSumBound}
    \omega\sum_{n=1}^{N} \sum_{l=1}^{m_n}\big\|Q_l^nx-Q_{l-1}^nx\big\|^2 \leq \|x-z\|^2-\|Ux-z\|^2 \leq 2\|Ux-x\|\,\|x-z\|,
  \end{equation}
  where $\omega = \min_{1\leq n \leq N} \omega_n$, because
  \begin{align}\nonumber
    \|Ux-z\|^2 & =\|Ux-x+x-z\|^2 \geq \|Ux-x\|^2 - 2\langle Ux-x,x-z\rangle + \|x-z\|^2 \\
               & \geq \|Ux-x\|^2 - 2\|Ux-x\|\|x-z\| + \|x-z\|^2 \geq - 2\|Ux-x\|\|x-z\| + \|x-z\|^2.
  \end{align}
  Setting $z=P_Cx$, we get
  \begin{equation}
  \omega \sum_{n=1}^{N} \sum_{l=1}^{m_n} \big\|Q_l^nx-Q_{l-1}^nx\big\|^2 \leq 2 \|Ux-x\|\,d(x,C).
  \end{equation}
  Given $x\in S$ and $i\in I$, choose $n\in\{1,\ldots,N\}$ and $1\leq p\leq m_n=|J_n|$, so that $i\in J_n$ and $Q_p^{n}=P_{C_i}Q_{p-1}^n$.
  Then we may conclude that
  \begin{align}\nonumber
    d(x,C_i)^2 & \leq \|x-P_{C_i}Q_{p-1}^nx\|^2 = \|x-Q_{p}^{n}x\|^2 \leq p \sum_{l=1}^{p} \|Q_l^nx-Q_{l-1}^nx\|^2\\
               & \leq m \sum_{n=1}^{N} \sum_{l=1}^{m_n} \|Q_l^nx-Q_{l-1}^nx\|^2 \label{eq:DistBound}
  \end{align}
  by using the triangle inequality, the convexity of the function $t\mapsto t^2$ and the fact that $m=\max_{1\leq n\leq N} m_n$. Combining~\eqref{eq:DistBound} with~\eqref{eq:DSumBound}, we obtain
  \begin{equation}\label{eq:FinalForStep2inLemma}
    d(x,C_i)^2 \leq \frac{2m}{\omega}(\|x-z\|^2-\|Ux-z\|^2) \leq \frac{2m}{\omega} \|Ux-x\| d(x,C),
  \end{equation}
  which finishes the proof of (ii).

  Finally, if the family $\{C_i\mid i\in I\}$ is $\kappa$-linearly regular over some $S\subseteq\mathcal{H}$, then~\eqref{eq:FinalForStep2inLemma} implies that
  \begin{equation}
  \frac{\omega}{2m\kappa^2} d(x,C)^2 \leq \frac{\omega}{2m} \max_{i\in I} d(x,C_i)^2
  \leq \|Ux-x\| d(x,C)
  \end{equation}
  holds for all $x\in S$, which proves the claim in (iv).
\end{proof}

\begin{theorem} \label{th:StringAveraging}
  Let $C_i\subseteq\mathcal{H}$, $i\in I:=\{1,\ldots,M\}$, be closed and convex, and assume $C:=\bigcap_{i\in I} C_i \neq \emptyset$. Let the sequence $\{x^k\}_{k=0}^\infty$ be defined by the dynamic string averaging projection method, that is, $x^0\in \mathcal H$ and for every
  $k=0,1,2,\ldots$, $x^{k+1}:=T_kx^k$, where
  \begin{equation}
  T_k:=\sum_{n=1}^{N_k} \omega^k_n \prod_{j\in J^k_n} P_{C_j},
  \end{equation}
  $J_n^k\subseteq I$ is a string and $\omega^k_n\in(0,1)$, $\sum_{n=1}^{N_k}\omega^k_n=1$. Moreover, let $I_k:=J^k_1\cup\ldots \cup J^k_{N_k}$.

  Assume that $\omega_n^k\geq \omega$ for some $\omega>0$ and that $m:=\sup_k\max_{1\leq n\leq N_k} |J_n^k|<\infty$. Moreover, assume that there is an integer $s\geq 1$ such that for each $k=0,1,2,\ldots$, we have $I=I_k\cup\ldots\cup I_{k+s-1}$.

  Then the following statements hold true:
  \begin{enumerate}[(i)]
    \item $\{x^k\}_{k=0}^\infty$ converges weakly to some point $x^\infty \in C\cap B(P_Cx^0,r)$, where $r=d(P_Cx^0,C)$.
    \item If the family $\{C_i\mid i\in I\}$ is regular over $B(P_C x^0,r)$, then the convergence is in norm.
    \item If the family $\{C_i\mid i\in I\}$ is $\kappa_r$-linearly regular over $B(P_Cx^0,r)$, then the convergence is linear, that is, $\|x^k-x^\infty\| \leq c_r q_r^k$ for each $k=0,1,2,\ldots,$ where
      \begin{equation}\label{eq:LinConvSA}
        c_r:=\frac{2d(x^0,C)}{q_r^{s-1}} \quad\text{and}\quad
        q_r := \sqrt[\leftroot{-1}\uproot{2}2s]{1-\frac{\omega}{2ms\kappa_r^2}}.
      \end{equation}
      Moreover, for each $k=0,1,2,\ldots,$ we have the following error bound:
      \begin{equation}\label{eq:ErrorSA}
        \frac{1}{2\kappa_r} \|x^k-x^\infty\| \leq \max_{i\in I} d(x^k,C_i) \leq c_r q_r^k.
      \end{equation}
    \end{enumerate}

\end{theorem}

\begin{proof}
  It is not difficult to see that statements (i) and (ii) are special cases of \cite[Theorem~4.1]{ReichZalas2016}. Indeed, since $\{x^k\}\subseteq B(P_Cx^0,r)$, there is no need to assume bounded regularity of $\{C_i\mid i\in I\}$ in~\cite[Theorem~4.1]{ReichZalas2016}. Consequently, for strong convergence,  one can repeat exactly the same argument as in the proof of~\cite[Theorem~4.1]{ReichZalas2016} while assuming regularity only over $B(P_Cx^0,r)$.
  Therefore it suffices to show statement~(iii). We divide the proof into three steps.

  \textbf{Step 1.}
  We first show that the inequality
  \begin{equation}\label{eq:Step1}
    d(x^{ks}, C_\nu)^2 \leq \frac{2ms}{\omega} \left(\|x^{ks}-z\|^2-\|x^{(k+1)s}-z\|^2\right)
  \end{equation}
  holds for every $\nu\in I$, $k\in\mathbb{N}$ and $z\in C$.

  Fix $\nu\in I$ and $z\in C$. Given $k\in\mathbb{N}$, we choose $l_k\in\{ks,\ldots,(k+1)s\}$ to be the smallest index so that $\nu\in I_{l_k}=\bigcup_{n=1}^{N_{l_k}}J_n^{l_k}$ and $n\in\{1,\ldots,N_{l_k}\}$ to be the smallest index such that $\nu\in J_n^{l_k}$, and set
  \begin{equation}
    i_k := \argmax_{j\in J_n^{l_k}} d(x^{l_k},C_j).
  \end{equation}
  Since the mapping $d(\cdot, C_\nu)$ is nonexpansive, we have
  \begin{equation}
    d(x^{ks},C_\nu) - d(x^{l_k},C_\nu) \leq |d(x^{ks},C_\nu) - d(x^{l_k},C_\nu)|\leq \|x^{ks}-x^{l_k}\|
  \end{equation}
  and hence
  \begin{equation}%\label{eq:DistNonExp}
    d(x^{ks},C_\nu)\leq \|x^{ks}-x^{l_k}\| + d(x^{l_k},C_\nu) \leq \|x^{ks}-x^{l_k}\| +  d(x^{l_k}, C_{i_k}),
  \end{equation}
  where the second inequality follows from the choice of $i_k$.
  Squaring this inequality, we obtain
  \begin{align}\label{eq:distsquared} \nonumber
    d(x^{ks},C_\nu)^2 & \leq \Big(\sum_{p=ks+1}^{l_k} \|x^{p}-x^{p-1}\| + d(x^{l_k}, C_{i_k})\Big)^2 \\ \nonumber
                    & \leq (l_k-ks+1) \Big(\sum_{p=ks+1}^{l_k} \|x^{p}-x^{p-1}\|^2 + d(x^{l_k}, C_{i_k})^2\Big) \\
                    & \leq s\; \Big(\sum_{p=ks+1}^{l_k} \|x^{p}-x^{p-1}\|^2
                      + \frac{2m}{\omega}\big(\|x^{l_k}-z\|^2-\|x^{l_k+1}-z\|^2\big)\Big)
  \end{align}
  by the Cauchy-Schwarz inequality, \eqref{eq:SAPartialBound} and since $l_k-ks+1\leq (k+1)s-ks = s$.
  Note that Theorem~\ref{thm:SQNECompConvComb} implies that $T_{p-1}$ is $\frac{1}{m}$-strongly quasi-nonexpansive and hence
  \begin{equation}
    \|x^p-x^{p-1}\|^2 = \|T_{p-1}x^{p-1}-x^{p-1}\|^2 \leq m (\|x^{p-1}-z\|^2-\|x^p-z\|^2),
  \end{equation}
  which, when combined with~\eqref{eq:distsquared}, leads to
  \begin{align} \nonumber
    d(x^{ks},C_\nu)^2 & \leq s \Big(\sum_{p=ks+1}^{l_k} m (\|x^{p-1}-z\|^2-\|x^p-z\|^2)
                        + \frac{2m}{\omega}\big(\|x^{l_k}-z\|^2-\|x^{l_k+1}-z\|^2\big)\Big) \\
                    &\leq \frac{2ms}{\omega} (\|x^{ks}-z\|^2-\|x^{l_k+1}-z\|^2)
                      \leq \frac{2ms}{\omega} (\|x^{ks}-z\|^2-\|x^{(k+1)s}-z\|^2),
  \end{align}
  where the last inequality holds since the sequence $\{x^{k}\}_{k=0}^{\infty}$ is Fej\'{e}r monotone with respect to $C$.

  \textbf{Step 2.} Setting $z=P_{C}x^{ks}$ in~\eqref{eq:Step1} and using
  \begin{equation}
    \|x^{(k+1)s}-P_{C}x^{(k+1)s}\|\leq  \|x^{(k+1)s}-P_{C}x^{ks}\|,
  \end{equation}
  we may deduce that
  \begin{equation}
    d(x^{ks}, C_\nu)^2 \leq \frac{2ms}{\omega} \left(d(x^{ks},C)^2-d(x^{(k+1)s},C)^2\right).
  \end{equation}
  Since $\nu\in I$ is chosen arbitrarily, we may use the $\kappa_r$-linear regularity of $\{C_i\mid i\in I\}$ over $B(P_Cx^0,r)$ to obtain
  \begin{equation}
  d(x^{ks},C)^2 \leq \kappa_r^2 \max_{\nu\in I} d(x^{ks},C_\nu)^2 \leq \frac{2ms\kappa_r^2}{\omega} \left(d(x^{ks},C)^2-d(x^{(k+1)s},C)^2\right),
  \end{equation}
  which we may rearrange to arrive at
  \begin{equation}
    d(x^{(k+1)s},C)^2 \leq \left(1-\frac{\omega}{2ms\kappa_r^2}\right) d(x^{ks},C)^2.
  \end{equation}
  Applying Theorem~\ref{thm:FejerMonotone}(\ref{thm:FejerMonotone:Item2}), we see that the subsequence $\{x^{kp}\}_{k=0}^{\infty}$ converges linearly to a point $x^\infty\in C$. In other words, for the point $x^\infty\in C$,  we get
  \begin{equation}
    \|x^{kp}-x^\infty\| \leq 2d(x_0,C) \left(1-\frac{\omega}{2ms\kappa_r^2}\right)^{k/2}.
  \end{equation}
  Finally, Theorem~\ref{thm:FejerMonotone}(\ref{thm:FejerMonotone:Item3}) implies that the whole sequence $\{x^{k}\}_{k=0}^{\infty}$ converges linearly too. More precisely, it implies that the sequence converges linearly with the constants as in~\eqref{eq:LinConvSA}.

  \textbf{Step 3.} Since $x^k\to x^\infty$, we can use Theorem~\ref{thm:FejerMonotone}(\ref{thm:FejerMonotone:Item1}) and the fact that $\{C_i\mid i\in I\}$ is $\kappa_r$-linearly regular to obtain
  \begin{equation}
    \|x^k-x^\infty\| \leq 2 d(x^k,C) \leq 2\kappa_r \max_{i\in I} d(x^k, C_i).
  \end{equation}
  On the other hand, using~$C\subseteq C_i$, $x^\infty\in C$ and Step~2, we get
  \begin{equation}
    \max_{i\in I} d(x^k,C_i) \leq d(x^k, C) \leq \|x^k-x^\infty\| \leq c_r q_r^k,
  \end{equation}
  which finishes the proof of the error bound~\eqref{eq:ErrorSA}.
\end{proof}

\subsection{Linear convergence of the cyclic and simultaneous projection methods}
In this section we specialize Theorem \ref{th:StringAveraging} (iii) to cyclic and simultaneous projection methods applied to general closed and convex sets. Results related to closed linear subspaces are discussed in Section \ref{sec:subspaces}. To this end, assume that the family $\{C_i\mid i\in I\}$ is $\kappa_r$-linearly regular over $B(P_Cx^0,r)$.

Consider first the cyclic projection method
\begin{equation}\label{eq:cyclicPMa}
  x^{0}\in\mathcal{H}\qquad\text{and}\qquad x^{k+1} := P_{C_{[k]}}x^k,
\end{equation}
where $[k]:= (k \operatorname{mod}M)+1$. In this case there is no convex combination, hence $\omega=1$. The length of the longest string $m=1$ and the cyclicality parameter $s=M$. Consequently, by Theorem \ref{th:StringAveraging} (iii),
\begin{equation}
q_r=\sqrt[2M]{1-\frac 1 {2M\kappa_r^2}}.
\end{equation}
The cyclic projection method oftentimes is considered in an equivalent formulation involving composition of projections, that is,
\begin{equation}\label{eq:cyclicPMb}
  x^{0}\in\mathcal{H}\qquad\text{and}\qquad x^{k+1} := P_{C_M}\ldots P_{C_1}x^k;
\end{equation}
see, for example, \cite[Theorem 4.5]{DeutschHundal2008} and \cite[Corollary 6.2]{BauschkeNollPhan2015}. Note that, in general, one can obtain the latter variant of the method by repeating the former one $M$ times. Thus, the coefficient expressing the linear rate becomes $q_r'=q_r^M$. In our case, this leads to
\begin{equation}\label{eq:ex:qr}
q_r'=\sqrt{1-\frac 1 {2M\kappa_r^2}}.
\end{equation}
On the other hand, the latter variant of the cyclic projection method is a particular case of the string averaging projection method with one string of length $m=M$. Consequently, by using Theorem \ref{th:StringAveraging} (iii) with $s=1$ we recover formula \eqref{eq:ex:qr}. The first linear convergence result for general closed and convex sets and the remotest-set projection method can be found in \cite[Theorem 1]{GurinPoljakRaik1967}. It is not difficult to see that in the case of $M=2$, the remotest-set projection method coincides with the cyclic projection. In addition, it was also reported in \cite{GurinPoljakRaik1967}, although without a proof, that the cyclic projection method may converge linearly even for $M>2$.

A similar derivation of $q_r$ can also be made for the simultaneous projection method
\begin{equation}\label{eq:simultaneousPM}
  x^{0}\in\mathcal{H}\qquad\text{and}\qquad x^{k+1} := \frac 1 M \sum_{i=1}^M P_{C_i}x^k.
\end{equation}
In this case the lower bound for the convex combination coefficients $\omega=\frac 1 M$, $m=1$ and $s=1$. Consequently, we obtain the same formula as in \eqref{eq:ex:qr}. The first result on the linear convergence rate of iteration \eqref{eq:simultaneousPM} in the case of general closed and convex sets is due to Pierra \cite[Theorem 1.1]{Pierra1984}, who established it in the equivalent product space setting using \cite[Theorem 1]{GurinPoljakRaik1967}.

In Table \ref{table} we compare several examples of the convergence rates related to the cyclic and simultaneous projection methods, which can be found in the literature. It is worth mentioning that in \cite{BauschkeBorwein1996, BauschkeNollPhan2015, BorweinLiTam2017,DeutschHundal2008} the formula for $q_r$ was not given explicitly in the statement of the theorem and, for the purpose of this paper, it was derived from the proof by using Theorem~\ref{thm:FejerMonotone}. On the other hand, in \cite{BauschkeBorwein1996, BauschkeNollPhan2015, BorweinLiTam2017}, as well as in this manuscript, the authors deal with more general schemes than cyclic and simultaneous projection methods, which might result in weaker estimates. Due to the equivalence of methods \eqref{eq:cyclicPMa} and \eqref{eq:cyclicPMb}, we focus only on the latter case.

\begin{table}[h]

\centering
\settowidth\rotheadsize{\theadfont superiorization}
\resizebox{0.9\linewidth}{!}{% Resize table to fit within \linewidth horizontally
\begin{tabular}{|c | c| c| c|}
\hline
\thead{Authors} &
\thead{Result}&
\thead{Cyclic PM\\$x^{k+1}=P_{C_M}\ldots P_{C_1}x^k$} &
\thead{Simultaneous PM\\$x^{k+1}=\frac 1 M \sum_{i=1}^M P_{C_i}x^k$}
\\

\hline %\hline
\makecell{Bauschke, Borwein, 1996} &
\cite[Theorem 5.7]{BauschkeBorwein1996} &
$q_r=\sqrt{1-\frac 1 {(2^M+2)\kappa_r^2}}$ &
$q_r=\sqrt{1-\frac 1 {(2+2M)\kappa_r^2}}$ \\

\hline %\hline
\makecell{Beck, Teboulle, 2003} &
\cite[Theorem 2.2]{BeckTeboulle2003} &
x &
$q_r=\sqrt{1-\frac 1 {M\kappa_r^2}}$ \\

\hline %\hline
\makecell{Deutsch, Hundal, 2008} &
\cite[Theorem 4.5]{DeutschHundal2008} &
$q_r=\sqrt{1-\frac 1 {4M\kappa_r^2}}$ &
x \\

\hline %\hline
\makecell{Bauschke, Noll,\\Phan, 2015} &
\cite[Theorem 6.1]{BauschkeNollPhan2015} &
$q_r=\sqrt{1-\frac 1 {2M\kappa_r^2}}$ &
$q_r=\sqrt{1-\frac 1 {(M+1)\kappa_r^2}}$ \\

\hline %\hline
\makecell{Borwein, Li,\\Tam, 2017} &
\cite[Theorem 3.1]{BorweinLiTam2017} &
$q_r=\sqrt{1-\frac 1 {2M\kappa_r^2}}$ &
$q_r=\sqrt{1-\frac 1 {(M+1)\kappa_r^2}}$ \\

\hline %\hline
\makecell{Kolobov, Reich,\\Zalas, 2017} &
\cite[Corollary 23]{KolobovReichZalas2016} &
$q_r=\sqrt{1-\frac 1 {M\kappa_r^2}}$ &
$q_r=\sqrt{1-\frac 1 {M\kappa_r^2}}$ \\

\hline %\hline
\makecell{Bargetz, Reich,\\Zalas} &
Theorem \ref{th:StringAveraging} &
$q_r=\sqrt{1-\frac 1 {2M\kappa_r^2}}$ &
$q_r=\sqrt{1-\frac 1 {2M\kappa_r^2}}$ \\

\hline
\end{tabular}
}%resizebox
\caption{
    Examples of $q_r$ for cyclic and simultaneous projection methods (PM) for which one has a linear convergence rate of the form $\|x^k-x^\infty\|\leq c_r q_r^k$. The estimates attributed to \cite{BauschkeBorwein1996, BauschkeNollPhan2015, BorweinLiTam2017, DeutschHundal2008} were deduced from the proofs by using Theorem~\ref{thm:FejerMonotone}.}
\label{table}

\end{table}

\subsection{Closed linear subspaces}\label{sec:subspaces}
Assume that each $C_i \subseteq\mathcal H$, $i\in I$, is a closed linear subspace. In this case the limit point $x^\infty=P_Cx^0$; see, for example, \cite[Proposition 5.9]{BauschkeCombettes2011}. Note that without any loss of generality we can assume that the family $ \{C_i\mid i\in I\}$ is $\kappa$-linearly regular over $\mathcal H$ in Theorem \ref{th:StringAveraging}. Indeed, we have the following proposition.
\begin{proposition}
  The family $\{C_i\mid i\in I\}$ is $\kappa$-linearly regular over $B(P_Cx^0,r)$ if and only if it is $\kappa$-linearly regular over $\mathcal H$.
\end{proposition}
\begin{proof}
  It sufficies to show that $\kappa$-linear regularity over the ball implies $\kappa$-linear regularity over $\mathcal H$. To this end, let $x\in\mathcal H$. Note that for any nonempty and closed linear subspace $M\subseteq\mathcal H$, $v\in M$, $\alpha\geq 0$ and $\beta\in \mathbb R$, we have $d(\alpha x+\beta v,M)=\alpha d(x,M)$. Thus,
  \begin{align}
    d(x,C) & = d(x-P_Cx^0, C)
    = \frac{\|x-P_Cx^0\|}{r}d\left(r \frac{x-P_Cx^0}{\|x-P_Cx^0\|},C\right)\\
    &\leq \frac{\|x-P_Cx^0\|}{r} \kappa \max_i d\left(r \frac{x-P_Cx^0}{\|x-P_Cx^0\|},C_i\right)= \kappa \max_i d(x,C_i),
  \end{align}
  where the inequality follows from the fact that $r(x-P_Cx^0)/\|x-P_Cx^0\|\in B(P_Cx^0,r)$.
  \end{proof}
  Consequently, the coefficients $\kappa_r$ and $q_r$ from Theorem \ref{th:StringAveraging} do not depend on $r$. Therefore it becomes reasonable to consider the smallest possible $\kappa$, which can be defined by
  \begin{equation}
  \kappa:=\supremum_{x\notin C}\frac{d(x,C)}{\max_{i\in I}d(x,C_i)}.
  \end{equation}
  The inverse of this constant $ \ell :=1/\kappa$, with the convention $[1/\infty=0]$, is called the \textit{inclination} of $\{C_i\mid i\in I\}$; see, for example, \cite{BadeaGrivauxMuller2011, PustylnikReichZaslavski2012, PustylnikReichZaslavski2013}.
  It turns out that the above $\kappa$ is strongly related to the extended cosine of the Friedrichs angle $\cos (\theta):=c(C_1,...,C_M)\in [0,1]$, where $\theta \in [0,\pi/2]$ and, following \cite[Definition 3.2]{BadeaGrivauxMuller2011},
  \begin{equation}\label{eq:def:Friedrich}
    c(C_1,...,C_M):=\sup\left\{\frac{1}{M-1}\frac{\sum_{i\neq j}\langle x_i,x_j\rangle }{\sum_{i=1}^M\|x_i\|^2}\mid x_i\in C_i\cap C^{\perp}\text{ and } \sum_{i=1}^M\|x_i\|^2\neq 0\right\}.
  \end{equation}
  As was shown in \cite[Lemma 3.1]{BadeaGrivauxMuller2011}, \eqref{eq:def:Friedrich} coincides with the classical definition of the cosine of the Friedrichs angle when $M=2$. Moreover, by \cite[Proposition 3.9]{BadeaGrivauxMuller2011}, we have
  \begin{equation}\label{eq:cosK}
    1-\frac{2M}{(M-1)\kappa}\leq \cos(\theta) \leq 1-\frac{1}{(N-1)\kappa^2}
  \end{equation}
  and, in particular, $\kappa <\infty$ if and only if $\cos(\theta)<1$. Thus the linear convergence rate of the dynamic string averaging projection method is related to the angle between the subspaces. Indeed, by using the first inequality, $1/\kappa^2\geq (\frac{1-\cos(\theta)}{4})^2$. Consequently, we can overestimate $q=q(\kappa,m,s,\omega)$ from \eqref{eq:LinConvSA} explicitly in terms of $\cos (\theta)$, that is,
  \begin{equation}\label{eq:QandCosSA}
    q(\kappa,m,s,\omega)=\sqrt[\leftroot{-1}\uproot{2}2s]{1-\frac{\omega}{2ms\kappa^2}}
    \leq \sqrt[\leftroot{-1}\uproot{2}2s]{1-\frac{\omega}{2ms}\left(\frac{1-\cos(\theta)}{4}\right)^2},
  \end{equation}
  which for the cyclic projection method $x^0\in\mathcal H;\ x^{k+1}:=P_{C_M}\ldots P_{C_1}x^k$, leads to
  \begin{equation}\label{eq:QandCyclic}
    q(\kappa,M,1,1)=\sqrt{1-\frac{1}{2M\kappa^2}}
    \leq \sqrt{1-\left(\frac{1-\cos(\theta)}{4M}\right)^2}=:r(\theta,M).
  \end{equation}
  This corresponds to \cite[Theorem 4.4]{BadeaGrivauxMuller2011}, where $\|(P_{C_M}\ldots P_{C_1})^k-P_C\|\leq r(\theta,M)^k$. As far as we know, the exact norm values, and thus the optimal convergence rates for the cyclic projection methods, are not known with the exception of $M=2$. Indeed, for the alternating projection method $x^0\in\mathcal H;\ x^{k+1}:=P_{C_2}P_{C_1}x^k$, the norm value due to Aronszajn \cite{Aronszajn1950} (inequality), and Kayalar and Weinert \cite{KayalarWeinert1988} (equality), we have
  $\|(P_{C_2}P_{C_1})^k-P_C\|=c(C_1,C_2)^{2k-1}.$ By using the second inequality from \eqref{eq:cosK}, we see that
  \begin{equation}
    c(C_1,C_2)^2 \leq q(\kappa,2,1,1)=\sqrt{1-\frac 1 {4\kappa^2}}.
  \end{equation}
  Note that the above inequality can also be derived directly by using the relation $1/\kappa = \sin (\theta/2)$ \cite[Theorem 3.7]{PustylnikReichZaslavski2012}. Indeed, we have $\cos^2 \theta\leq \cos^2(\theta/2)=1-\sin^2(\theta/2)\leq\sqrt{1-\frac 1 {4\kappa^2}}$.

\section{Perturbation resilience and superiorization}\label{sec:InexactOrbits}
In this section we consider the question of perturbation resilience of infinite products of nonexpansive operators. We show that the rate of convergence, both for weak as well as for strong convergence, is essentially preserved under summable perturbations. These results are applicable, in particular, to the string averaging projection methods discussed in Section \ref{sec:StringAveraging}. We present this connection in more detail in Section \ref{sec:StringAveraging2}.

\begin{theorem}[Perturbation resilience]\label{th:GSC} % general stable convergence
  For every $k=0,1,2,\ldots$, let $T_k\colon\mathcal H\rightarrow\mathcal H$ be nonexpansive and assume that $C\subseteq \bigcap_k \fix T_k$ is nonempty, closed and convex. Let $\{e_k\}_{k=0}^\infty\subseteq[0,\infty)$ be a given summable sequence of errors and let $\{x^k\}_{k=0}^\infty\subseteq\mathcal H$ be an inexact trajectory such that
  \begin{equation}\label{eq:th:GSC:ek}
    \|x^{k+1}-T_kx^k\|\leq e_k
  \end{equation}
  holds for every $k=0,1,2,\ldots$. Assume that for all $i=0,1,2,\ldots$, there is an $x_i^\infty\in C$ such that
  \begin{equation}\label{eq:WeakConvAss}
    T_k\cdots T_i x^i \rightharpoonup x_i^\infty.
  \end{equation}
  Then there is a point $x^\infty\in C$ such that $x_i^\infty\to x^\infty$ and $x^k\rightharpoonup x^\infty$. Moreover,
  \begin{equation}\label{eq:WeakRate}
    |\langle y, x^{k+1}-x^\infty\rangle| \leq |\langle y, T_k\cdots T_i x^i-x_i^\infty\rangle|+ 2 \|y\| \sum_{k=i}^{\infty} e_k
  \end{equation}
  holds for all $y\in\mathcal{H}$ and all $i=0,1,2,\ldots$. If, in addition, for all $i=0,1,2,\ldots$,
  \begin{equation}\label{eq:StrongConvAss}
    T_k\cdots T_ix^i \to x_i^\infty,
  \end{equation}
  then $x^k\to x^\infty$. Moreover,
  \begin{equation}\label{eq:StrongRate}
    \|x^{k+1}-x^\infty\| \leq \|T_k\ldots T_i x^i - x_i^\infty\| + 2 \sum_{k=i}^\infty e_k.
  \end{equation}
\end{theorem}

\begin{proof}
  We claim that for every integer $k\geq i \geq 0$, we have
  \begin{equation}\label{eq:th:GSC:proof:step1}
    \|x^k-x^k_i\|\leq \sum_{j=i}^{k-1} e_j,
  \end{equation}
  where $x^{k+1}_i:=T_k\ldots T_ix^i$ and $x^i_i=x^i$. Note that by the definition of $x^k_i$, this inequality is true for every integer $k=i\geq 0$ with the right-hand side equal to zero. Now fix $i$ and assume that $k> i$. Thus, by the triangle inequality, (\ref{eq:th:GSC:ek}) and the nonexpansivity of $T_k$, we have
  \begin{align}
    \|x^{k+1}-x^{k+1}_i\| & \leq \|x^{k+1}+T_kx^k\| + \|T_kx^k-T_kx^k_i\| \nonumber\\
                          & \leq e_k+\|x^k-x^k_i\|,
  \end{align}
  which, by induction, yields (\ref{eq:th:GSC:proof:step1}).

  Now we can apply Theorem~\ref{thm:WeakPertRes} to get $x^k\rightharpoonup x^\infty$. Therefore, for any $y\in\mathcal H$, we may take the limit as $k\to\infty$ in
  \begin{equation}
    |\langle y, x^{k+1}-x_i^\infty\rangle| \leq |\langle y, x_i^{k+1}-x_i^\infty\rangle|+\|y\| \|x_i^{k+1}-x^{k+1}\|
    \leq  |\langle y, T_k\ldots T_ix^i-x_i^\infty\rangle|+\|y\| \sum_{j=i}^{k} e_j
  \end{equation}
  and using~\eqref{eq:WeakConvAss}, arrive at
  \begin{equation}\label{eq:ImprBound}
    |\langle y, x^{\infty}-x_i^\infty\rangle| \leq \|y\| \sum_{j=i}^{\infty} e_j.
  \end{equation}
  Consequently, $x_i^\infty\rightharpoonup x^\infty$. Note that~\eqref{eq:ImprBound} implies that $x_i^\infty\to x^\infty$ as well. Indeed, we have
  \begin{equation}
    \|x^\infty-x_i^\infty\| = \sup_{\|y\|\leq 1} |\langle y, x^\infty-x_i^\infty\rangle| \leq \sum_{j=i}^{\infty} e_j
  \end{equation}
  and by letting $i\to\infty$, we get strong convergence.

  Moreover, by the triangle inequality, \eqref{eq:th:GSC:proof:step1} and~\eqref{eq:ImprBound}, we have the following estimate:
  \begin{align}\label{eq:FinalWeak} \nonumber
    |\langle y, x^{k+1}-x^\infty\rangle| & \leq |\langle y, x^{k+1}-x_i^{k+1}\rangle| +
                                           |\langle y, x_i^{k+1}-x_i^\infty\rangle|+ |\langle y, x_i^\infty-x^\infty\rangle| \\ \nonumber
                                         & \leq \|y\| \|x^{k+1}-x_i^{k+1}\|+|\langle y, x_i^{k+1}-x_i^\infty\rangle|+\|y\|  \sum_{j=i}^{\infty} e_j\\
                                         & \leq 2\|y\| \sum_{j=i}^\infty e_j + |\langle y, x_i^{k+1}-x_i^\infty\rangle|,
  \end{align}
  which proves~\eqref{eq:WeakRate}.

  Now assume that $x_i^k\to x_i^\infty$. Since convergence in norm implies weak convergence, we may use~\eqref{eq:FinalWeak} to obtain
  \begin{equation}
      |\langle y, x^{k+1}-x^\infty\rangle| \leq 2 \sum_{j=i}^\infty e_j + \|x_i^{k+1}-x_i^\infty\|
  \end{equation}
  for all $y\in \mathcal{H}$ with $\|y\|\leq 1$ and hence
  \begin{equation}
    \|x^{k+1}-x^\infty\|=\sup_{\|y\|\leq 1} |\langle y, x^{k+1}-x^\infty\rangle| \leq 2 \sum_{j=i}^\infty e_j + \| x_i^{k+1}-x_i^\infty\|,
  \end{equation}
  which proves $x^k\to x^\infty$ and (\ref{eq:StrongRate}).
\end{proof}

\begin{remark}
  Note that the proof of Theorem~\ref{th:GSC} also works in the case of Banach spaces if we replace $y\in\mathcal{H}$ by a bounded linear functional. Hence all the results in this section are also true in the setting of Banach spaces. Since in this article we are only interested in operators on Hilbert spaces, we formulate the theorem only for this case.
\end{remark}

\begin{remark}\label{rem:WeakConvRate}
  Clearly the sequence $\{\|T_k\ldots T_ix^i-x_i^\infty\|\}_{k=i}^\infty$ measures the convergence rate of the unperturbed method, where the exact computations start from the iterate $x^i$. The same could apply to the sequence $\{|\langle y, T_k\cdots T_i x^i-x_i^\infty\rangle|\}_{k=i}^\infty$. Indeed, if we use the term ``weak convergence rate'' in the sense that for every $y\in\mathcal{H}$ there is a mapping $r_y\colon\mathbb{N}\to(0,\infty)$ such that
  \begin{equation}
    |\langle y, x^k-x^\infty\rangle |\leq r_y(k),
  \end{equation}
  then Theorem~\ref{th:GSC} can be interpreted as follows: both for weak and strong convergence, the convergence rate is essentially preserved in the presence of summable perturbations. Note that since $d(x^{k+1},C)\leq\|x^{k+1}-x^\infty\|$, by~\eqref{eq:StrongRate} and Lemma~\ref{lem:DistToSet}, we may also deduce estimates concerning the distance of the sequence to the set $C_\varepsilon$. Indeed, given $\varepsilon>0$, there is $i_\varepsilon \geq 0$ such that
  \begin{equation}
    d(x^{k+1}, C) \leq \|T_k\ldots T_ix^i-x^\infty_i\| + \varepsilon
  \end{equation}
  for $k\geq i \geq i_\varepsilon$. For these $k$, either $x^{k+1}\in C_\varepsilon$ and the inequality $d(x^{k+1},C_\varepsilon) \leq\|T_k\ldots T_i x^i - x_i^\infty\|$ trivially holds true, or if $x^{k+1}\not\in C_\varepsilon$, we may then obtain the inequality
  \begin{equation} \label{eq:cor:GSC:estim3}
    d(x^{k+1},C_\varepsilon)\leq \|x^{k+1}-x^\infty\|-\varepsilon\leq\|T_k\ldots T_i x^i - x_i^\infty\|
  \end{equation}
  by using Lemma~\ref{lem:DistToSet}.
\end{remark}

\begin{corollary}[Superiorization]\label{th:Super} 
  For every $k=0,1,2,\ldots$, let $T_k\colon\mathcal H\rightarrow\mathcal H$ be nonexpansive and assume that $C\subseteq \bigcap_k \fix T_k$ is nonempty, closed and convex. Given a summable sequence of positive numbers $\{\beta_k\}_{k=0}^\infty$ and a bounded sequence $\{v^k\}_{k=0}^\infty$ in $\mathcal{H}$, we define for $x^{0}\in\mathcal{H}$ the sequence $\{x^k\}_{k=0}^{\infty}$ by
  \begin{equation}\label{eq:Superiorization}
    x^{k+1} = T_{k}(x^k-\beta_k v^{k}).
  \end{equation}
  Assume that for every $i=0,1,2,\ldots$, there is $x_i^\infty \in C$ such that
  \begin{equation}\label{eq:th:Super:assumpt}
    T_k\cdots T_i x^i \rightharpoonup x_i^\infty.
  \end{equation}
  Then there is a point $x^\infty\in C$ such that $x_i^\infty\to x^\infty$ and $x^k\rightharpoonup x^\infty$. Moreover,
  \begin{equation}\label{eq:th:Super:WeakRate}
    |\langle y, x^{k+1}-x^\infty\rangle| \leq |\langle y, T_k\cdots T_i x^i-x_i^\infty\rangle|+ 2 \|y\| \sum_{k=i}^{\infty} \beta_k\|v^k\|
  \end{equation}
  holds for all $y\in\mathcal{H}$ and all $i=0,1,2,\ldots$. If, in addition, for all $i=0,1,2,\ldots$,
  \begin{equation}
    T_k\cdots T_ix^i \to x_i^\infty,
  \end{equation}
  then $x^k\to x^\infty$. Moreover,
  \begin{equation}\label{eq:th:Super:StrongRate}
    \|x^{k+1}-x^\infty\| \leq \|T_k\ldots T_i x^i - x_i^\infty\| + 2 \sum_{k=i}^\infty \beta_k\|v^k\|.
  \end{equation}
\end{corollary}
\begin{proof}
  By the nonexpansivity of $T_k$, we have $\|x^{k+1}-T_kx^k\|\leq \beta_k\|v_k\|$. Thus it suffices to substitute $e_k:=\beta_k \|v^k\|$ and apply Theorem \ref{th:GSC}.
\end{proof}

\begin{remark}
  As we have already mentioned in the Introduction, the steering sequence $\{v^k\}_{k=0}^\infty$ is related to a subgradient of some convex, continuous function $\phi\colon\mathcal H\rightarrow \mathbb R$. For example, one can use $v^k\in\partial \phi(x^k)$. 
  A variant of~\eqref{eq:Superiorization} can be considered, where $v^k$ is replaced by a linear combination, that is,
  \begin{equation}\label{eq:VarSuperorization}
    x^{k+1} = T_{k}\Big(x^k-\sum_{n=0}^{L_k} \beta_{k,n}v^{k,n}\Big),
  \end{equation}
  the sequence of all $v^{k,n}$ is bounded, $L_k\leq L$ for a fixed integer $L$  and the coefficients satisfy the summability condition
  \begin{equation}
    \sum_{k=0}^{\infty} \sum_{n=0}^{L_k} \beta_{k,n} < \infty.
  \end{equation}
  This can be found, for example, in \cite[Algorithm 4.1]{CZ2015Superiorization}. In this paper the authors choose a norm-one vector $v^{k,n}\in\partial \phi(x^k-\sum_{l=0}^{n-1}\beta_{k,l}v^{k,l})$ whenever $0\notin \partial \phi(x^k-\sum_{l=0}^{n-1}\beta_{k,l}v^{k,l})$ and set $v^{k,n}=0$ otherwise, where $n=0,\ldots,L_k$.
  For simplicity we restrict ourselves to the case where, as in~\eqref{eq:Superiorization}, at each step there is only a single $v^k$. Also note that by setting
  \begin{equation}
    \beta_k := \sum_{n=0}^{L_k} \beta_{k,n}\qquad\text{and}\qquad v^k := \sum_{n=0}^{L_k} \frac{\beta_{k,n}}{\beta_k} v^{k,n},
  \end{equation}
  we can rewrite~\eqref{eq:VarSuperorization} in the form of~\eqref{eq:Superiorization}, where $\|v^k\|\leq\max_{1\leq n\leq L_k} \|v^{k,n}\|$.
\end{remark}

\begin{example}[Preservation of the linear rate]\label{ex:LRPertSup} % stable convergence linear rate
  Let us consider the basic method $x^{k+1}=T_kx^k$, where $T_k\colon\mathcal H\rightarrow\mathcal H$ are nonexpansive such that $\emptyset\neq C\subseteq \fix T_k$. Assume that this method converges linearly, that is, for given starting points $x^i\in \mathcal H$, $i=0,1,2,\ldots,$ there are $c_i\in(0,\infty)$, $q_i\in(0,1)$ and $x_i^\infty \in C$ such that for every $k\geq i$,
  \begin{equation}\label{eq:ex:LRPertSup:LR}
    \|T_k\ldots T_i x^i - x_i^\infty\| \leq c_i\cdot q_i^{k-i}.
  \end{equation}
  We show that this linear rate can be preserved first, by adding perturbations and second, by considering a superiorized version of the basic iterative method.
  \begin{enumerate}[(a)]
  \item (Perturbation resilience) To this end, let $\{x^k\}_{k=0}^\infty$ be a perturbed trajectory of the basic method with summable perturbations $\{e_k\}_{k=0}^\infty\subseteq [0,\infty)$, that is, $\|T_kx^k-x^{k+1}\|\leq e_k$ for every $k=0,1,2,\ldots$; compare with Theorem \ref{th:GSC}. Then we have the following estimate:
    \begin{equation}\label{eq:ex:LRPertSup:estim1}
      d(x^{k+1},C)\leq \|x^{k+1}-x^\infty\| \leq  c_i \cdot q_i^{k-i} + 2 \sum_{k=i}^\infty e_k.
    \end{equation}
    In particular, for every $\varepsilon>0$, there is an index $i_\varepsilon \geq 0$ such that the inequality
    \begin{equation}\label{eq:ex:LRPertSup:estim2}
      d(x^{k+1},C_\varepsilon)\leq  c_i \cdot q_i^{k-i}
    \end{equation}
    holds for every $k\geq i\geq i_\varepsilon$ similarly to the argument in Remark~\ref{rem:WeakConvRate}.

  \item (Superiorization) Now assume that $\{x^k\}_{k=0}^\infty$ is a superiorized version of the basic method, that is, $x^0\in\mathcal H$, $x^{k+1} = T_{k}(x^k-\beta_k v^{k})$, for every $k=0,1,2,\ldots$, where $\{\beta_k\}_{k=0}^\infty\subseteq [0,\infty)$ is summable and  $\{v^k\}_{k=0}^\infty \subseteq\mathcal{H}$ is bounded; compare with Corollary~\ref{th:Super}. Then, similarly to case (a), estimates~\eqref{eq:ex:LRPertSup:estim1} and~\eqref{eq:ex:LRPertSup:estim2} hold true with $e_k=\beta_k\|v^k\|$.

  \end{enumerate}
\end{example}

Note that in some cases we have more information on the subset $C$, for example, $C=\bigcap_{i=1}^M C_i$, where every $C_i$ is closed and convex; see Section~\ref{sec:StringAveraging}. In this case it would be of interest to control the quantity
\begin{equation}
  \max_{i=1,\ldots, M} d(x^k,C_i),
\end{equation}
which measures in some sense the distance of $x^k$ from $x^\infty$; see, for example, Theorem~\ref{th:StringAveraging} and estimate~\eqref{eq:ErrorSA}. Note that estimate~\eqref{eq:ErrorSA} is valid only if no perturbations are taken into account. The question which arises now is whether we can expect a similar estimate in the presence of perturbations. To answer this question we proceed with the following general theorem.

\begin{theorem}[Perturbation resilience continued]\label{th:GSC:Ci}
  Assume that all the assumptions from Theorem~\ref{th:GSC} are satisfied. Then there is $R>0$ such that for every $i=0,1,2,\ldots,$ we have $\{T_k\ldots T_i x^i\}_{k=i}^\infty\subseteq B(0,R)$. If, in addition, $C=\bigcap_{i\in I} C_i$, where $C_i\subseteq \mathcal H$ are closed and convex, $i\in I:=\{1,\ldots,M\}$, and the family $\{C_i\mid i\in I\}$ is $\kappa$-linearly regular over $B(0,R)$, then for every integer $k\geq i\geq 0$, we have
  \begin{equation}\label{eq:th:GSC:Ci:estimate}
    \frac{1}{2\kappa}\|x^{k+1}-x^\infty\| - 2\sum_{k=i}^\infty e_k\leq \max_{j\in I}d(x^{k+1},C_j)
    \leq\|T_k\ldots T_i x^i - x_i^\infty\| + 2 \sum_{k=i}^\infty e_k.
  \end{equation}
\end{theorem}
\begin{proof}
For every integer $k\geq i\geq 0$, let $x^{k+1}_i:=T_k\ldots T_ix^i$ and $x^i_i:=x^i$. By Theorem~\ref{th:GSC}, $\{x^k\}_{k=0}^\infty$ is convergent and thus bounded. Note that $\{x^k_i\}_{k=i}^\infty$ is Fej\'{e}r monotone with respect to $C$. Therefore, for every integer $k\geq i$ and $z\in C$, we have
\begin{equation}
  \|x^k_i-z\|\leq \|x^i_i-z\|=\|x^i-z\|
\end{equation}
which proves that $\{x^k_i\}_{k=i}^\infty\subseteq B(0,R)$ for every $i=0,1,2,\ldots,$ and some $R>0$.

Now we show that estimate (\ref{eq:th:GSC:Ci:estimate}) holds true. Since the sequence $\{x^k_i\}_{k=i}^\infty$ converges to some $x^\infty_i\in C$, by Theorem~\ref{thm:FejerMonotone}~(i), we have
\begin{equation}\label{eq:th:GSC:Ci:proof:1}
\|x^k_i-x^\infty_i\|\leq 2 d(x^k_i,C)
\end{equation}
for every integer $k\geq i \geq 0$. Moreover, by (\ref{eq:th:GSC:proof:step1}),
\begin{equation}\label{eq:th:GSC:Ci:proof:2}
  \|x^k-x^k_i\|\leq \sum_{k=i}^\infty e_k.
\end{equation}
Thus, by (\ref{eq:StrongRate}), (\ref{eq:th:GSC:Ci:proof:1}), linear regularity over $B(0,R)$ and the definition of the metric projection, we get
\begin{align}\label{eq:th:GSC:Ci:proof:3} \nonumber
  \|x^{k+1}-x^\infty\|&\leq \|x^{k+1}_i-x^\infty_i\|+2\sum_{k=i}^\infty e_k\\ \nonumber
                      &\leq 2d(x^{k+1}_i,C)+2\sum_{k=i}^\infty e_k\\ \nonumber
                      &\leq 2\kappa \max_{j\in I}d(x^{k+1}_i,C_j)+2\sum_{k=i}^\infty e_k\\
                      &\leq 2\kappa \max_{j\in I}\|x^{k+1}_i-P_{C_j}x^{k+1}\|+2\sum_{k=i}^\infty e_k,
\end{align}
which by the triangle inequality, (\ref{eq:th:GSC:Ci:proof:2}) and again by the definition of the metric projection is less than or equal to
\begin{align}\label{eq:th:GSC:Ci:proof:4} \nonumber
  & 2\kappa \max_{j\in I}d(x^{k+1},C_j)+2\kappa\|x^{k+1}_i-x^{k+1}\|+ 2\sum_{k=i}^\infty e_k\\ \nonumber
  &\leq 2\kappa \max_{j\in I}d(x^{k+1},C_j)+ 4\kappa\sum_{k=i}^\infty e_k\\ \nonumber
  &\leq 2\kappa d(x^{k+1},C)+ 4\kappa\sum_{k=i}^\infty e_k\\ \nonumber
  &\leq 2\kappa \|x^{k+1}-x^\infty_i\|+ 4\kappa\sum_{k=i}^\infty e_k\\ \nonumber
  &\leq 2\kappa\left(\|x^{k+1}_i-x_i^\infty\|+2\sum_{k=i}^\infty e_k\right)+4\kappa\sum_{k=i}^\infty e_k\\
  &\leq 2\kappa\|x^{k+1}_i-x_i^\infty\|+8\kappa\sum_{k=i}^\infty e_k.
\end{align}
Thus (\ref{eq:th:GSC:Ci:proof:3}) and (\ref{eq:th:GSC:Ci:proof:4}) complete the proof.
\end{proof}

It is easy to see that one can formulate a similar result in terms of superiorization. Indeed, we have the following corollary.

\begin{corollary}[Superiorization continued]\label{th:Super:Ci}
  Assume that all the assumptions from Corollary \ref{th:Super} are satisfied. Then there is $R>0$ such that for every $i=0,1,2,\ldots,$ we have $\{T_k\ldots T_i x^i\}_{k=i}^\infty\subseteq B(0,R)$. If, in addition, $C=\bigcap_{i\in I} C_i$, where $C_i\subseteq \mathcal H$ are closed and convex, $i\in I:=\{1,\ldots,M\}$, and the family $\{C_i\mid i\in I\}$ is $\kappa$-linearly regular over $B(0,R)$, then for every i=0,1,2,\ldots and $k\geq i$, we have
  \begin{equation}
    \frac{1}{2\kappa}\|x^{k+1}-x^\infty\| - 2\sum_{k=i}^\infty \beta_k\|v^k\|\leq \max_{j\in I}d(x^{k+1},C_j)
    \leq\|T_k\ldots T_i x^i - x_i^\infty\| + 2 \sum_{k=i}^\infty \beta_k\|v^k\|.
  \end{equation}
\end{corollary}

\section{String averaging projection methods revisited}\label{sec:StringAveraging2}

In this section we revisit string averaging projection methods and combine the results of the previous sections in order to obtain results regarding the convergence rate of dynamic string averaging projection methods in the presence of perturbations. As a corollary, we consider the convergence rate of superiorized dynamic string averaging methods.

\begin{theorem}[Perturbation resilience] \label{th:SA:Inexact}
  Let $C_i\subseteq\mathcal{H}$, $i\in I:=\{1,\ldots,M\}$, be closed and convex, and assume that $C:=\bigcap_{i\in I} C_i \neq \emptyset$. Let $\{T_k\}_{k=0}^\infty$ be a sequence of operators defined as in Theorem~\ref{th:StringAveraging}. Let the sequence $\{x^k\}_{k=0}^\infty$ be an inexact orbit of the dynamic string averaging projection method defined by $\{T_k\}_{k=0}^\infty$, that is, the inequality
  \begin{equation}
    \|x^{k+1}-T_kx^k\|\leq e_k
  \end{equation}
  holds for every $k=0,1,2,\ldots$, where the sequence $\{e_k\}_{k=0}^\infty\subseteq[0,\infty)$ is summable.

  Then the following statements hold true:
  \begin{enumerate}[(i)]
  \item For every $i=0,1,2,\ldots,$ there is $x_i^\infty \in C$ such that $T_k\cdots T_i x^i \rightharpoonup_k x_i^\infty \rightarrow_i x^\infty$ for some $x^\infty\in C$. Moreover, for all $y\in\mathcal{H},$ we have
    \begin{equation}
      |\langle y, x^{k+1}-x^\infty\rangle| \leq |\langle y, T_k\cdots T_i x^i-x_i^\infty\rangle|+ 2 \|y\| \sum_{k=i}^{\infty} e_k.
    \end{equation}

  \item If the family $\{C_i\mid i\in I\}$ is boundedly regular, then for every $i=0,1,2,\ldots,$ there is $x_i^\infty \in C$ such that $T_k\cdots T_i x^i \rightarrow_k x_i^\infty \rightarrow_i x^\infty$ for some point $x^\infty\in C$. Moreover, we have
    \begin{equation}
      \|x^{k+1}-x^\infty\| \leq \|T_k\ldots T_i x^i - x_i^\infty\| + 2 \sum_{k=i}^\infty e_k.
    \end{equation}
  \item If the family $\{C_i\mid i\in I\}$ is boundedly linearly regular, then there are $\kappa>0$, $c>0$ and $q>0$, and $x^\infty\in C$ such that
    \begin{equation}
      \frac{1}{2\kappa}\|x^{k+1}-x^\infty\| - 2\sum_{k=i}^\infty e_k\leq \max_{j\in I}d(x^{k+1},C_j)
      \leq c q^{k-i} + 2 \sum_{k=i}^\infty e_k.
    \end{equation}
  \end{enumerate}
\end{theorem}
\begin{proof}
  Combine Theorem \ref{th:StringAveraging}, Theorem \ref{th:GSC} and Theorem \ref{th:GSC:Ci}.
\end{proof}

\begin{corollary}[Superiorization] \label{th:SA:Super}
  Let $C_i\subseteq\mathcal{H}$, $i\in I:=\{1,\ldots,M\}$, be closed and convex, and assume that $C:=\bigcap_{i\in I} C_i \neq \emptyset$. Let $\{T_k\}_{k=0}^\infty$ be a sequence of operators defined as in Theorem~\ref{th:StringAveraging}. Let the sequence $\{x^k\}_{k=0}^\infty$ be a superiorized version of the dynamic string averaging projection method defined by $\{T_k\}_{k=0}^\infty$, that is,
  \begin{equation}
    x^0\in\mathcal H\qquad\text{and}\qquad x^{k+1} = T_{k}(x^k-\beta_k v^{k}),
  \end{equation}
  where $\{\beta_k\}_{k=0}^\infty\subseteq[0,\infty)$ is summable and $\{v^k\}_{k=0}^\infty\subseteq\mathcal H$ is bounded.

  Then the following statements hold true:
  \begin{enumerate}[(i)]
  \item for every $i=0,1,2,\ldots,$ there is $x_i^\infty \in C$ such that $T_k\cdots T_i x^i \rightharpoonup_k x_i^\infty \rightarrow_i x^\infty$
    for some $x^\infty\in C$. Moreover, for all $y\in\mathcal{H}$, we have
    \begin{equation}
      |\langle y, x^{k+1}-x^\infty\rangle| \leq |\langle y, T_k\cdots T_i x^i-x_i^\infty\rangle|+ 2 \|y\| \sum_{k=i}^{\infty} \beta_k\|v^k\|;
    \end{equation}

  \item If the family $\{C_i\mid i\in I\}$ is boundedly regular, then for every $i=0,1,2,\ldots,$ there is $x_i^\infty \in C$ such that $T_k\cdots T_i x^i \rightarrow_k x_i^\infty \rightarrow_i x^\infty$ for some point $x^\infty\in C$. Moreover, we have
    \begin{equation}
      \|x^{k+1}-x^\infty\| \leq \|T_k\ldots T_i x^i - x_i^\infty\| + 2 \sum_{k=i}^\infty \beta_k\|v^k\|.
    \end{equation}
  \item If the family $\{C_i\mid i\in I\}$ is boundedly linearly regular, then there are $\kappa>0$, $c>0$ and $q>0$, and $x^\infty\in C$ such that
    \begin{equation}
      \frac{1}{2\kappa}\|x^{k+1}-x^\infty\| - 2\sum_{k=i}^\infty \beta_k\|v^k\|\leq \max_{j\in I}d(x^{k+1},C_j)
      \leq c q^{k-i} + 2 \sum_{k=i}^\infty \beta_k\|v^k\|.
    \end{equation}
  \end{enumerate}
\end{corollary}

\begin{proof}
  Combine Theorem \ref{th:StringAveraging}, Corollary \ref{th:Super} and Corollary \ref{th:Super:Ci}.
\end{proof}

\vspace{3em}
\noindent\textbf{Acknowledgment.} The authors are grateful to the anonymous referees for all their comments and remarks which helped improve this paper.

\vspace{2em}
\noindent\textbf{Funding.} 
This research was supported in part by the Israel Science Foundation (Grant~389/12), the Fund for the Promotion of Research at the Technion and by the Technion General Research Fund.
Open access funding was provided by University of Innsbruck and Medical University of Innsbruck.

\small
%\bibliographystyle{spmpsci}
%\bibliography{references}

\end{document}